\documentclass[11pt]{amsart}
\usepackage{amsfonts, amsmath, amssymb, color}
\usepackage{amsthm}
\usepackage{booktabs}
\usepackage{hhline}
\usepackage{lipsum}
\usepackage{lscape}
\usepackage{subfig}
\usepackage{textcomp}
\usepackage{tikz}
\usetikzlibrary{decorations.pathmorphing,shapes,arrows,positioning}
\usepackage{xcolor}
\usepackage{young}
\usepackage{youngtab}
\usepackage{ytableau}

\allowdisplaybreaks[4]
\theoremstyle{plain}
\newtheorem{thm}{Theorem}[section]
\newtheorem{lem}[thm]{Lemma}
\newtheorem{prop}[thm]{Proposition}
\newtheorem{cor}[thm]{Corollary}
\newtheorem{rem}[thm]{Remark}

\theoremstyle{definition}

\newtheorem{exmp}[thm]{Example}

\newcommand{\la}{\lambda}

\numberwithin{equation}{section} \errorcontextlines=0

\begin{document}
\title{On irreducible characters of the Iwahori-Hecke algebra in type $A$}
\author{Naihuan Jing}
\address{Department of Mathematics, North Carolina State University, Raleigh, NC 27695, USA}
\email{jing@ncsu.edu}
\author{Ning Liu}
\address{School of Mathematics, South China University of Technology,
Guangzhou, Guangdong 510640, China}
\email{mathliu123@126.com}

\maketitle

\begin{abstract} We use vertex operators to compute irreducible characters of the Iwahori-Hecke algebra of type $A$. Two general formulas are given for the irreducible characters in terms of those of the symmetric groups or the Iwahori-Hecke algebras in lower degrees.
Explicit formulas are derived for the irreducible characters labeled by hooks and two-row partitions. Using duality, we 
also formulate a determinant type Murnaghan-Nakayama formula and give another proof of Ram's combinatorial Murnaghan-Nakayama formula.
As applications, we study super-characters of the Iwahori-Hecke algebra as well as the bitrace of the regular representation and provide a simple proof of the Halverson-Luduc-Ram formula.
\end{abstract}

\section{Introduction}

Let $\mathbb{C}(q)$ be the field of rational functions in the variable $q$. The Iwahori-Hecke algebra $H_{n}(q)$ is the unital associative algebra
over $\mathbb{C}(q)$ generated by generators $T_{1}, T_{2},\ldots, T_{n-1}$ subject to the relations
\begin{align}\label{e:Hecke1}
T_{i}T_{j}&=T_{j}T_{i},  ~~~~\text{if}~ |i-j|>1,\\ \label{e:Hecke2}
T_{i}T_{i+1}T_{i}&=T_{i+1}T_{i}T_{i+1},\\ \label{e:Hecke3}
T_{i}^{2}&=(q-1)T_{i}+q.
\end{align}

The algebra $H_n(q)$ becomes a Frobenius algebra under the non-degenerate associative bilinear form $B(h_1, h_2)=\chi(h_1h_2)$, where
$\chi(\sum_{h\in H_n(q)}a_h h)=a_1$. The algebra $H_{n}$ is semisimple \cite{Ho, W} and isomorphic to $\mathbb C[\mathfrak S_n]$, so its irreducible representations are also labeled
by partitions of $n$. Let $\chi^{\lambda}$ be the irreducible character of $H_n(q)$ associated with partition $\lambda\vdash n$.
Ram \cite{Ram} used the quantum Schur-Weyl duality to prove the Frobenius type character formula
for $H_n(q)$ in terms of one-row Hall-Littlewood functions and Schur symmetric functions (see also \cite{KW, KV}). He also gave the $q$-analogue Murnaghan-Nakayama formula for $H_n(q)$ as an iterative combinatorial
rule to compute the irreducible characters in terms of those corresponding to smaller partitions.

Since Ram's work, there have been various discussions and generalizations on the rule, cf. \cite{vdJ, Ha, P, G, Ro, Sh} etc.
A nice computer algebra program \cite{GHLMP} on Hecke algebras
is also available.
All these developments have shown that the $q$-Murnaghan-Nakayama rule remains the most
practical algorithm to compute the irreducible characters of $H_n(q)$ (also see Starkey's rule \cite{St, G}).
Nevertheless, how to effectively carry out the computation deserves further study; getting new formulas and
reformulating the known ones and related structures may also help practical computation and offer new perspective.

The goal of this paper is to offer an efficient and practical method to compute irreducible characters of $H_n(q)$ using vertex operators. In the vertex operator approach to symmetric functions \cite{Jing1, Jing3}, Hall-Littlewood and Schur symmetric functions are expressed as simple products of vertex operators. Unlike the usual raising operators in symmetric functions which are often not well-defined, the vertex operators
obey nice algebraic structures, so the irreducible character values can be studied in terms of the
vertex operators and their dual operators.
Using this idea, we first derive two general formulas to compute irreducible character values of the Hecke algebra: the first
one computes the character in terms of those of
the symmetric group of lower degrees, while the second one reduces the computation to lower degree characters of the Hecke algebra. We emphasize that
both general formulas are different from the $q$-Murnaghan-Nakayama formula.
Thanks to these new formulas, we are able to give explicit compact formulas of the irreducible $H_n(q)$-characters labeled by hooks and two-row partitions. The formulas are then utilized to give a simple proof of
the $q$-analogue Berele-Regev formula for $H_n(q)$, which was first obtained \cite{Zhao} by using the iterative Murnaghan-Nakayama formula.

Using the same strategy we 
formulate a determinant type Murnaghan-Nakayama rule for $H_n(q)$ by exploiting the Jacobi-Trudi rule. The determinantal version implies the combinatorial Murnaghan-Nakayama rule easily. Finally we compute the bitrace of the regular representation of the Iwarhoti-Hecke algebra, originally computed by Halverson, Luduc and Ram \cite{HLR} using Roichman's formula. Our method is a straightforward computation using the vertex operator techniques.

The structure of the paper is as follows. Section 2 discusses how to treat irreducible characters by 
the vertex operator realization of Schur and
 Hall-Littlewood symmetric functions. We express all irreducible characters of the Hecke algebra of type $A_{n-1}$
 as matrix coefficients of vertex operators
 in \cite{Jing1}. Based on this we derive two general formulas to compute all irreducible characters of the Iwahori-Hecke algebra, including
 one to express the characters in terms of those of the symmetric group in lower degrees. We derive compact formulas for the irreducible characters corresponding to hooks and two-row partitions
 (see \eqref{e:hook}-\eqref{e:two}). In Section 3, we first use the general vertex operator formula to formulate
a determinant type Murnaghan-Nakayama formula for $H_n(q)$, which gives another proof of the combinatorial one.
Our current approach to the problem is based upon the idea of dual vertex operators developed in \cite{Jing1} and \cite{Jing4}, which was first used in \cite{Jing3} on Schur's Q-functions.

\section{Vertex operators and character values $\chi^{\la}_{\mu}(q)$}
Let $\Lambda$ be the ring of symmetric functions in the $x_n$ ($n\in\mathbb N$) over the integers. In this paper we mostly work with
the ring $\Lambda_F$ over the field $F=\mathbb Q(t)$ and view it as a graded ring under the natural degree.

The ring $\Lambda$ has several linear bases indexed by partitions.
A {\it partition} $\lambda=(\lambda_1,\lambda_2,\ldots)$ is a weakly decreasing sequence of non-negative integers. The sum $|\lambda|=\sum_i\lambda_i$ is called the weight
 and the number of nonzero parts is called the length $l(\lambda)$.
A partition $\lambda$ of weight $n$ is denoted by
$\lambda \vdash n$, and the
set of partitions is denoted by $\mathcal P$. When the parts of $\lambda$ are arranged in increasing order, so $\lambda=(1^{m_{1}}2^{m_{2}}\ldots)$
where $m_{i}=Card\{\lambda_{j}=i\mid 1\leq j\leq l(\lambda)\}$
is the multiplicity of $i$ in $\lambda$. $\la$ is called a {\it strict partition} if $m_i\leq 1$. When the finite sequence $\lambda=(\lambda_1,\lambda_2,\ldots)$ of nonnegative integers
is an ordered or
not necessarily weakly decreasing one such that $\sum_i\la_i=n$, $\lambda$ is called a {\it composition} of $n=\sum_i\lambda_i$, denoted as $\lambda\models n$. The length
$l(\lambda)$ is the number of nonzero parts.

A partition $\lambda$ is visualized by its Young diagram: the set of nodes (or boxes situated at) $(i,j)\in \mathbb{Z}_+^{2}$ such that $1\leq j\leq \lambda_{i}$. If $\lambda$ is a diagram, then an inner corner of $\lambda$ is a node $(i,j)\in \lambda$ whose removal still leaves the diagram as
that of a partition. The conjugate partition $\la'=(\la_1', \ldots, \la_{\la_1}')$ corresponds to the reflection of
the Young diagram along the diagonal.

For each $r>1$, let $p_{r}=\sum x_{i}^{r}$ be the $r$th power-sum. Then $p_{\lambda}=p_{\lambda_{1}}p_{\lambda_{2}}\cdots p_{\lambda_{l}} (\lambda\in\mathcal P)$ form a $\mathbb Q$-basis of $\Lambda_{\mathbb Q}$. Let $s_{\lambda}$ be the
Schur function associated with the partition $\lambda$, then $s_{\lambda}$'s form an orthonormal basis $\Lambda$ under the inner product
\begin{align}\label{e:inner}
\langle p_{\lambda}, p_{\mu}\rangle=\delta_{\lambda\mu}z_{\lambda}
\end{align}
where $z_{\lambda}=\prod_{i\geq 1}\lambda_im_i(\lambda)!$. The function $p_n$ acts on $\Lambda$ as a multiplication operator, and its
the dual operator 
is the differential operator $p_n^* =n\frac{\partial}{\partial p_n}$. Note that * is $\mathbb Q(t)$-linear and  anti-involutive satisfying
\begin{equation}
\langle p_nu, v\rangle=\langle u, p_n^*v\rangle
\end{equation}
for $u, v\in \Lambda$.

We now recall the vertex operator realization of the Schur symmetric functions \cite{Jing3}.
Let
$S(z)$ and the dual vertex operator $S^*(z)$ be the linear maps: $\Lambda\longrightarrow \Lambda[[z, z^{-1}]]$ defined by
\begin{align}\label{e:Schurop}
S(z)&=\mbox{exp} \left( \sum\limits_{n\geq 1} \dfrac{1}{n}p_nz^{n} \right) \mbox{exp} \left( -\sum \limits_{n\geq 1} \frac{\partial}{\partial p_n}z^{-n} \right)=\sum_{n\in\mathbb Z}S_nz^{n},\\ \label{e:Schurop*}
S^*(z)&=\mbox{exp} \left(-\sum\limits_{n\geq 1} \dfrac{1}{n}p_nz^{n} \right) \mbox{exp} \left(\sum \limits_{n\geq 1} \frac{\partial}{\partial p_n}z^{-n} \right)=\sum_{n\in\mathbb Z}S^*_nz^{-n}.
\end{align}
We use the convention to index the components by their degrees.
The operators $S_n \in \mathrm{End}(\Lambda)$ are the Bernstein vertex operators realizing the Schur functions. The dual operators
$S_n^*\in \mathrm{End}(\Lambda)$, introduced in \cite{Jing1}, also realize the Schur functions.

The following relations will be useful in our discussion.
\begin{prop}\cite{Jing1} (1) The components of $S(z)$ and $S^{*}(z)$ obey the following commutation relations:
\begin{align}\label{e:relation}
S_{m}S_{n}+S_{n-1}S_{m+1}&=0,\\
S^{*}_{m}S^{*}_{n}+S^{*}_{n+1}S^{*}_{m-1}&=0,\\
S_{m}S^{*}_{n}+S^{*}_{n-1}S_{m-1}&=\delta_{m,n}.
\end{align}
(2) For any composition $\mu=(\mu_{1},\ldots,\mu_{k})$, the
product $S_{\mu_{1}}\cdots S_{\mu_{k}}.1=s_{\mu}$ is the
Schur function labeled by $\mu$. In general, $s_{\mu}=0$ or $\pm s_{\lambda}$ for a partition $\lambda$ such that $\lambda\in \mathfrak{S}_{l}(\mu+\delta)-\delta.$ Here $\delta=(l-1,l-2,\ldots,1,0),$ where $l = l(\mu)$.
Moreover, $S_{-n}.1=\delta_{n,0}, S^{*}_{n}.1=\delta_{n,0}, (n\geq0)$.
\end{prop}

Let $q_{n}=q_{n}(x;t)$ be the generalized homogeneous symmetric function defined by
\begin{align}\label{e:qop}
Q(z)=\mbox{exp} \left( \sum\limits_{n=1}^{\infty}\frac{1-t^{n}}{n}p_{n}z^{n} \right)=\sum\limits_{n\geq0}q_{n}z^{n},
\end{align}
and $q_{n}=0$ if $n<0$. Denote
$q_{\lambda}=q_{\lambda_{1}}q_{\lambda_{2}}\cdots q_{\lambda_{l}}$
for any partition $\lambda$, then the set $\{q_{\lambda}\}$ forms a basis of $\Lambda_{\mathbb{Q}(t)}$.
Their dual operators with respect to the inner product \eqref{e:inner} are defined by
\begin{align}\label{e:qop*}
Q^{*}(z)=\mbox{exp} \left( \sum\limits_{n=1}^{\infty}(1-t^{n})\frac{\partial}{\partial p_{n}}z^{-n} \right)=\sum\limits_{n\geq0}q_{n}^{*}z^{-n},
\end{align}
and $q_n^*=0$ if $n<0$. In particular, $q^*_n.1=\delta_{n, 0}$ for $n\geq 0$.

For $n\geq0,$ by \eqref{e:qop} and \eqref{e:Schurop*}, we have
\begin{align}
q_{n}&=\sum\limits_{\lambda\vdash n}\frac{1}{z_{\lambda}(t)}p_{\lambda},\\
S^*_{-n}.1&=\sum\limits_{\lambda\vdash n}\frac{(-1)^{l(\lambda)}}{z_{\lambda}}p_{\lambda},
\end{align}
where $z_{\lambda}(t)=\prod\limits_{i\geq1}\frac{i^{m_{i}(\lambda)}m_{i}(\lambda)!}{1-t^{\lambda_{i}}},$ and $z_{\lambda}=z_{\lambda}(0).$

Let $H_n(q)$ be the Iwahori-Hecke algebra of type $A$ (see \eqref{e:Hecke1}-\eqref{e:Hecke3}).
Let $w=s_{i_1}\cdots s_{i_k}$ be a reduced expression of $w\in \mathfrak S_n$, where $s_i=(i,i+1)$. We define $T_w=T_{i_1}\cdots T_{i_k}$, which is
well-defined and independent from the choice of reduced expressions.
Then $H_n(q)$ has a linear basis consisting of $T_w$, $w\in \mathfrak S_n$.

For $\sigma_r=(12\cdots r)\in \mathfrak S_n$ in cycle notation, let $T_{\sigma_r}$ be the corresponding element in $H_n(q)$. For
 any composition $\mu=(\mu_1, \ldots, \mu_l)$ of $n$, let $\sigma_{\mu}=\sigma_{\mu_1}\times \cdots \times \sigma_{\mu_l}\in \mathfrak S_{\mu_1}\times \cdots\times \mathfrak S_{\mu_l}\hookrightarrow \mathfrak S_n$, we define $T_{\sigma_{\mu}}=T_{\sigma_{\mu_1}}\cdots T_{\sigma_{\mu_l}}$.
Let $\phi: H_n(q)\longrightarrow \mathrm{End}(V)$ be an irreducible representation associated with partition $\lambda$. In general the character $\chi^{\lambda}(T)=\mathrm{Tr}(\phi(T))$ is
no longer a function of conjugacy classes of $S_n$. However, it is known that all irreducible character values at elements $T_w$ are determined by their values at the element $T_{\sigma_{\mu}}$ \cite{GP}. For this reason, we will denote $\chi^{\lambda}(T_{\sigma_{\mu}})=\chi^{\lambda}_{\mu}(q)$.

Denote $\tilde{q}_{r}(t)=\frac{t^{r}}{t-1}q_{r}(t^{-1})$ and let
$\tilde{q}_{\mu}(t)=\tilde{q}_{\mu_{1}}(t)\tilde{q}_{\mu_{2}}(t)\cdots\tilde{q}_{\mu_{l}}(t)=\frac{t^{|\mu|}}{(t-1)^{l(\mu)}}q_{\mu}(t^{-1}).$
We have the following Frobenius type formula for the characters of $H_{n}(q)$ from \cite{Ram}.

\begin{prop} \cite{Ram} The irreducible character $\chi^{\lambda}$ of $H_{n}$ corresponding to $\lambda$ is determined by
\begin{align}\label{e:characters}
\tilde{q}_{\mu}(q)=\sum\limits_{\lambda\vdash n}\chi^{\lambda}_{\mu}(q)s_{\lambda}.
\end{align}
where $\mu\vdash n$ and $s_{\lambda}$ is the Schur function associated with partition $\lambda$.
\end{prop}

Therefore, 
\begin{align}\label{e:chi}
\chi^{\lambda}_{\mu}(q)&=\langle \tilde{q}_{\mu}(q), s_{\lambda}\rangle
=\frac{q^{|\mu|}}{(q-1)^{l(\mu)}}\langle q_{\mu}(q^{-1}), S_{\lambda}.1\rangle,
\end{align}
and we are going to compute $g^{\lambda}_{\mu}(t):=\langle q_{\mu}(t), S_{\lambda}.1\rangle$ in the following.

\begin{prop}\label{Up Down}
For any $m\in\mathbb Z_+, n\in\mathbb{Z}$
\begin{align}\label{e:com1}
S^{*}_{n}q_{m}=q_{m}S^{*}_{n}+(1-t)\sum\limits_{k=1}^{m}q_{m-k}S^*_{n-k},\\ \label{e:com2}
q^{*}_{m}S_{n}=S_{n}q^{*}_{m}+(1-t)\sum\limits_{k=1}^{m}S_{n-k}q^*_{m-k}.
\end{align}
\end{prop}
\begin{proof} The usual vertex operator calculus gives that
\begin{align}
\label{e:relations1}
S^*(z)Q(w)&=Q(w)S^*(z)\frac{z-tw}{z-w},
\\ \label{e:relations2}
Q^{*}(w)S(z)&=S(z)Q^{*}(w)\frac{w-tz}{w-z},
\end{align}
where the rational functions are expanded at $w=0$ and $z=0$ respectively. The relations then follow by comparing coefficients of $z^{-n}w^{m}$ in \eqref{e:relations1} and $z^{n}w^{-m}$ in \eqref{e:relations2} respectively.
\end{proof}

For two compositions $\lambda,\mu,$ we say $\lambda\subset \mu$ if $\lambda_{i}\leq \mu_{i}$ for all $i\geq1$. In this case,
we write $\lambda-\mu=(\lambda_{1}-\mu_{1},\lambda_{2}-\mu_{2},\ldots)\vDash |\la|-|\mu|$. For each partition $\la=(\la_1, \ldots, \la_l)$, we define that
\begin{align}
\la^{[i]}=(\la_{i+1}, \cdots, \la_l), \qquad i=0, 1, \ldots, l
\end{align}
So $\la^{[0]}=\la$ and $\la^{[l]}=\emptyset$.
Next, we give the main results.

\begin{thm}\label{t:iterative}
For partitions $\lambda,\mu\vdash n$ and integer number $k$,
\begin{align}\label{e:qS}
q_{k}^{*}S_{\lambda}.1&=\sum\limits_{\tau\models k}(1-t)^{l(\tau)}S_{\lambda-\tau}.1\\ \label{e:Sq}
S^{*}_{k}q_{\mu}&=\sum_{\mbox{\tiny$\begin{array}{c}
\tau\subset \mu\\ |\tau|\geq k 
\end{array}$}}(1-t)^{l(\tau)}q_{\mu-\tau}S^{*}_{k-\mid \tau \mid}.1
=\sum\limits_{i=k}^{n}\sum\limits_{\tau\in \mathcal{C}^{\mu}_{i}}(1-t)^{l(\tau)}q_{\mu-\tau}S^{*}_{k-i}.1,
\end{align}
where $\mathcal{C}^{\mu}_{k}\triangleq\{\tau\models k\mid \tau\subset\mu\}.$
\end{thm}
\begin{proof}
We argue by induction on $l(\lambda)$ for the first relation. The initial step is clear. Assume that
\eqref{e:qS} holds for any partition with length $<l(\lambda)$,
it follows from Proposition \ref{Up Down} that
\begin{align*}
q_{k}^{*}S_{\lambda}&=S_{\lambda_{1}}q^{*}_{k}S_{\lambda_{2}}\cdots S_{\lambda_{l}}.1+(1-t)\sum\limits_{i=1}^{k}S_{\lambda_{1}-i}q^*_{k-i}S_{\lambda_{2}}\cdots S_{\lambda_{l}}.1\\
&=\sum\limits_{\tau\models k}(1-t)^{l(\tau)}S_{\lambda_{1}}S_{\lambda^{[1]}-\tau}.1+\sum\limits_{i=1}^{k}\sum\limits_{\tau\models k-i}(1-t)^{l(\tau)+1}S_{\lambda_{1}-i}S_{\lambda^{[1]}-\tau}.1\\
&=\sum\limits_{\tau\models k}(1-t)^{l(\tau)}S_{\lambda-\tau}.1.
\end{align*}
The other relation \eqref{e:Sq} is shown similarly.
\end{proof}

\begin{exmp}
Let $\lambda=(321), \mu=(2^21^2)$ be two partitions. 
\begin{align*}
g^{\lambda}_{\mu}(t)&=\langle q_{2}q_{2}q_{1}q_{1}, S_{3}S_{2}S_{1}.1 \rangle\\
&=(1-t)^{2}\langle q_{2}q_{1}q_{1}, S_{3}S_{1}.1+S_{2}S_{2}.1+S_{2}S_{1}S_{1}.1 \rangle\\
&=2(1-t)^4\langle q_{1}q_{1}, S_{1}S_{1}.1+S_{2}.1 \rangle\\
&=4(1-t)^5\langle q_{1}, S_{1}.1 \rangle\\
&=4(1-t)^6.
\end{align*}
\end{exmp}

To compute the characters of the Hecke algebra, we collect some simple facts (see \cite{JL}).
\begin{lem}\label{l:z} For partitions $\lambda,\mu \vdash n,$
\begin{align}\label{e:q.1}
\sum\limits_{\lambda \vdash n}\frac{(-1)^{l(\lambda)}}{z_{\lambda}(t)}&=
\begin{cases}
t^n-t^{n-1}&\text{if $n\geq 1$}\\
1&\text{if $n=0$}
\end{cases}
\\ \label{e:S.1}
\sum\limits_{\lambda \vdash n}\frac{1}{z_{\lambda}(t)}&=
\begin{cases}
1-t&\text{if $n\geq 1$}\\
1&\text{if $n=0$}
\end{cases}
\\ \label{e:X(1^n)}
{\chi}^{(1^n)}_{\mu}(1)&=\langle p_{\mu}, S_{(1^n)}.1 \rangle=(-1)^{n-l(\mu)}
\\ \label{e:X(n)}
{\chi}^{(n)}_{\mu}(1)&=\langle p_{\mu}, S_n.1 \rangle=1.
\end{align}
\end{lem}

The following formulas  are well-known \cite{Ram}. They are also obtained as special cases in our formulas for the
hook and two-row cases. 
\begin{prop} For $\mu\vdash n$, one has that
\begin{align}\label{e:onerow}
\chi^{(n)}_{\mu}(q)&=q^{n-l(\mu)},\\ \label{e:onecolumn}
\chi^{(1^n)}_{\mu}(q)&=(-1)^{n-l(\mu)}.
\end{align}
\end{prop}

For $\mu\vdash n$ and $1\leq i\leq n$, we introduce two sequences of polynomials respectively in $t$ and $t^{-1}$:
\begin{align}\label{e:ai}
a_i(\mu;t)&=(1-t)^{-l(\mu)}\sum\limits_{\tau\in \mathcal{C}^{\mu}_{i}}(1-t)^{l(\mu-\tau)}(1-t^{-1})^{l(\tau)},\\ \label{e:bi}
b_i(\mu;t)&=(1-t^{-1})^{-l(\mu)}\sum\limits_{\tau\in \mathcal{C}^{\mu}_{i}}(1-t^{-1})^{l(\mu-\tau)+l(\tau)}.
\end{align}

They are fixed by the generating functions shown below.
\begin{lem} \label{l:ab} For a partition $\mu=(\mu_1, \ldots, \mu_r)$, we have that
\begin{align}\label{e:av}
(1-t^{-1}v)^r\prod_{i=1}^r(1+v+\cdots+v^{\mu_i-1})&=\sum_{i=0}^{|\mu|}a_i(\mu, t)v^i\\ \label{e:bv}
\prod_{i=1}^r(t^{-1}+v^{\mu_i}+(1-t^{-1})(1+v+\cdots+v^{\mu_i-1}))&=\sum_{i=0}^{|\mu|}b_i(\mu, t)v^i
\end{align}
\end{lem}
\begin{proof} Let $c_i(\mu; t')=(1-t')^{-r}\sum\limits_{\tau\in \mathcal{C}^{\mu}_{i}}(1-t')^{l(\mu-\tau)}(1-t^{-1})^{l(\tau)}$.
Then 
\begin{align*}
(1-t')^r\sum_{i=0}^{|\mu|}c_i(\mu;t')v^i&=\sum_{k=0}^{|\mu|}\sum_{\tau\subset\mu, \tau\vDash k}(1-t')^{l(\mu-\tau)}(1-t^{-1})^{l(\tau)}v^k\\
&=\sum_{k=0}^{|\mu|}\sum_{\tau_1+\cdots+\tau_r=k}(1-t')^{r-\delta_{\tau_1,\mu_1}-\cdots -\delta_{\tau_r,\mu_r}}(1-t^{-1})^{r-\delta_{\tau_1,0}-\cdots-\delta_{\tau_r,0}}v^k\\
&=(1-t')^r\prod_{i=1}^r\left(\sum_{\tau_i=0}^{\mu_i}(1-t')^{-\delta_{\tau_i,\mu_i}}(1-t^{-1})^{1-\delta_{\tau_i,0}}v^{\tau_i}\right),
\end{align*}
which implies the lemma by noting that 
\begin{gather*}
1+\sum_{i=1}^{n-1}(1-t^{-1})v^i+\frac{1-t^{-1}}{1-t'}v^n=\begin{cases}
(1-t^{-1}v)[n]_v, & t'=t\\
t^{-1}+v^k+(1-t^{-1})[n]_v, & t'=t^{-1}\end{cases}.
\end{gather*}
\end{proof}

We remark that $a_0(\mu; t)=b_0(\mu; t)=1$ and for any $0\leq j\leq n=|\mu|$
\begin{align*}
a_j(\mu; t)&=(-1)^{l(\mu)}a_{n-j}(\mu;t^{-1}),\\
b_j(\mu;t)&=b_{n-j}(\mu;t).
\end{align*}

Now we give the irreducible character formulas labeled by hook and two-row partitions. 
\begin{thm}\label{t:hook}
For partition $\mu\vdash n$ and $k\geq 0$, we have
\begin{align}\label{e:hook}
\chi^{(k,1^{n-k})}_{\mu}(q)=(-1)^{n-k+l(\mu)}\sum\limits_{i=k}^{n}a_i(\mu;q)q^i
\end{align}
\end{thm}
\begin{proof} Let $l=l(\mu)$. By Theorem \ref{t:iterative} 
it follows that
\begin{align*}
&g^{(k,1^{n-k})}_{\mu}(t)=\langle q_{\mu}, S_{(k,1^{n-k})}.1 \rangle=\langle S_{k}^{*}q_{\mu}, S_{(1^{n-k})}.1 \rangle\\
=&\sum\limits_{i=k}^{n}\sum_{\mbox{\tiny$\begin{array}{c}
\mid \tau \mid=i\\
\tau\subset \mu\end{array}$}}(1-t)^{l(\tau)}\langle q_{\mu-\tau}S^*_{k-i}.1, S_{(1^{n-k})}.1 \rangle\\
=&\sum\limits_{i=k}^{n}\sum_{\mbox{\tiny$\begin{array}{c}
\mid \tau \mid=i\\
\tau\subset \mu\end{array}$}}(1-t)^{l(\tau)}\sum\limits_{\mu^{(1)}\vdash \mu_{1}-\tau_{1}, \cdots, \mu^{(l)}\vdash \mu_{l}-\tau_{l}}\sum\limits_{\lambda\vdash i-k}\frac{(-1)^{l(\lambda)}}{z_{\mu^{(1)}}(t)\cdots z_{\mu^{(l)}}(t)z_{\lambda}}
\langle p_{\mu^{(1)}\cup\cdots\cup\mu^{(l)}\cup\lambda}, S_{(1^{n-k})}.1 \rangle\\
=&\sum\limits_{i=k}^{n}\sum_{\mbox{\tiny$\begin{array}{c}
\mid \tau \mid=i\\
\tau\subset \mu\end{array}$}}(1-t)^{l(\tau)}\sum\limits_{\mu^{(1)}\vdash \mu_{1}-\tau_{1}}\cdots\sum\limits_{\mu^{(l)}\vdash \mu_{l}-\tau_{l}}\sum\limits_{\lambda\vdash i-k}\frac{(-1)^{n-k}(-1)^{l(\mu^{(1)})}\cdots(-1)^{l(\mu^{(l)})}}{z_{\mu^{(1)}}(t)\cdots z_{\mu^{(l)}}(t)z_{\lambda}}\\
=&(-1)^{n-k}\sum\limits_{i=k}^{n}\sum\limits_{\tau\in \mathcal{C}^{\mu}_{i}}t^{n-i}(1-t^{-1})^{l(\mu-\tau)}(1-t)^{l(\tau)},
\end{align*}
where we have used \eqref{e:q.1} and \eqref{e:S.1} in the last equation. The theorem follows by recalling \eqref{e:ai}
and $\chi_{\mu}^{\la}(q)=(-1)^{l(\mu)}q^{n}(1-q)^{-l(\mu)}g_{\mu}^{\la}(q^{-1})$.
\end{proof}

Lemma \ref{l:ab} offers practical way to compute \eqref{e:hook} by expanding the generating function up to
certain power of $v$ and then setting $v=q$.
\begin{exmp} Let $\lambda=(6,1^2),\mu=(2^4)$. Taking terms of $v$ with degree $\geq 6$ and then setting $v=q$,
\begin{align*}
(1-t^{-1}v)^4(1+v)^4&\equiv(1-t^{-1}v)^4(6v^2+4v^3+v^4)\\
&\equiv t^{-4}v^4\cdot 6v^2+(t^{-4}v^4-4t^{-3}v^3)4v^3+(t^{-4}v^4-4t^{-3}v^3+6t^{-2}v^2)v^4\\
&\equiv 6t^2-12t^3+3t^4.
\end{align*}
Therefore $\chi^{(6, 1^2)}_{(2^4)}(q)=6q^2-12q^3+3q^4$.
\end{exmp}

\begin{thm}\label{t:two}
For partition $\mu\vdash n$ and $k\geq n-k\geq 0$ we have
\begin{align}\label{e:two}
\chi^{(k,n-k)}_{\mu}(q)&=q^{n-l(\mu)}(b_k(\mu;q)-b_{k+1}(\mu;q)),\\
\sum_{i=0}^{[n/2]}\chi^{(n-i, i)}_{\mu}(q)&=q^{n-l(\mu)}b_{[n/2]}(\mu;q).
\end{align}
\end{thm}
\begin{proof} By the same argument in the proof of Theorem \ref{t:hook} we see that $g^{(k,n-k)}_{\mu}(t)=\langle S_k^*q_{\mu}, S_{n-k}.1\rangle$
is equal to 
\begin{align*}
&\sum\limits_{i=k}^{n}\sum\limits_{\tau\in\mathcal{C}^{\mu}_{i}}(1-t)^{l(\tau)}\sum\limits_{\mu^{(1)}\vdash \mu_{1}-\tau_{1}, \cdots, \mu^{(l)}\vdash \mu_{l}-\tau_{l}}\sum\limits_{\lambda\vdash i-k}\frac{(-1)^{l(\lambda)}}{z_{\mu^{(1)}}(t)\cdots z_{\mu^{(l)}}(t)z_{\lambda}}\\
=&\sum\limits_{\tau\in\mathcal{C}^{\mu}_{k}}(1-t)^{l(\mu-\tau)+l(\tau)}-\sum\limits_{\tau\in\mathcal{C}^{\mu}_{k+1}}(1-t)^{l(\mu-\tau)+l(\tau)},
\end{align*}
which implies the first result by \eqref{e:bi}. The second one is clear.
\end{proof}

\begin{exmp}
Let $\lambda=(4,2),\mu=(3,2,1)$. By Theorem \ref{t:two} 
\begin{align*}
&\prod_{i=1}^3(t^{-1}+v^i+(1-t^{-1})[i]_v)-v^{-1}\prod_{i=1}^3(t^{-1}+v^i+(1-t^{-1})[i]_v)\\
&=(v-v^{-1})(t^{-1}+v^2+(1-t^{-1})(1+v))(t^{-1}+v^3+(1-t^{-1})(1+v+v^2)).
\end{align*}
The coefficient of $v^4$ is $(1-t^{-1})(2-t^{-1})$. Therefore
$$\chi_{(321)}^{(42)}(q)=q^3(b_4(q)-b_5(q))=q^3(2-3q^{-1}+q^{-2})=2q^3-3q^2+q.$$
\end{exmp}

Using the similar argument, we can compute the general character values.

For a partition-valued function $\underline{\mu}=(\mu^{(1)}, \ldots, \mu^{(r)})$, we define
\begin{equation}
z_{\underline{\mu}}(t)=z_{\mu^{(1)}}(t)\cdots z_{\mu^{(r)}}(t),
\end{equation}
and call $|\underline{\mu}|=\sum_i|\mu^{(i)}|$ the weight of $\underline{\mu}$ and $l(\underline{\mu})=\sum_il(\mu^{(i)})$ the length of $\underline{\mu}$.
Given a composition $\tau=(\tau_1, \ldots, \tau_r)$, we denote by $\underline{\mu}\vdash \tau$ the partition-valued function $\underline{\mu}=(\mu^{(1)}, \ldots, \mu^{(r)})$
such that $\mu^{(i)}\vdash \tau_i$. Clearly $|\underline{\mu}|=|\tau|$.

The following general formula expresses the Hecke algebra character in terms of irreducible characters of the symmetric groups of lower degrees.
\begin{thm}\label{t:gen} For $\la, \mu\vdash n$, the irreducible character $\chi_{\mu}^{\la}(q)$ is given by
\begin{align}
\chi^{\la}_{\mu}(q)=\frac1{(q-1)^{l(\mu)}}\sum_{i=\la_1}^nq^i\sum_{\mbox{\tiny$\begin{matrix}|\tau|=i\\
\tau\subset\mu\end{matrix}$}}(1-q^{-1})^{l(\tau)}
\sum_{{\underline\nu}\vdash \mu-\tau}\sum_{\rho\vdash i-\la_1}\frac{(-1)^{l(\underline{\nu})+l(\rho)}}
{z_{\underline{\nu}}(q)z_{\rho}}\chi^{\la^{[1]}}_{\underline{\nu}\cup\rho}.
\end{align}
\end{thm}
\begin{proof} Let $l(\mu)=l$. As in the proof Theorem \ref{t:hook} we have that
\begin{align*}
&\chi^{\la}_{\mu}(q)=\frac{q^n}{(q-1)^l}\langle S_{\la_1}^{*}q_{\mu}, S_{\la^{[1]}}.1 \rangle\\
=&\frac{q^n}{(q-1)^l}\sum\limits_{i=\la_1}^{n}\sum_{\mbox{\tiny$\begin{array}{c}
\mid \tau \mid=i\\
\tau\subset \mu\end{array}$}}(1-q^{-1})^{l(\tau)}\sum_{\underline{\nu}\vdash \mu-\tau}\sum\limits_{\rho\vdash i-\la_1}\frac{(-1)^{l(\rho)}}{z_{\underline{\nu}}(q^{-1})z_{\rho}}\chi^{\la^{[1]}}_{\underline{\nu}\cup\rho}\\
=&\frac1{(q-1)^l}\sum\limits_{i=\la_1}^{n}\sum_{\mbox{\tiny$\begin{array}{c}
\mid \tau \mid=i\\
\tau\subset \mu\end{array}$}}q^i(1-q^{-1})^{l(\tau)}\sum_{\underline{\nu}\vdash \mu-\tau}\sum\limits_{\rho\vdash i-\la_1}\frac{(-1)^{l(\rho)+l(\underline{\nu})}}{z_{\underline{\nu}}(q)z_{\rho}}\chi^{\la^{[1]}}_{\underline{\nu}\cup\rho}.
\end{align*}
\end{proof}

We remark that one obtains another general formula for $\chi_{\mu}^{\la}(q)$ based on the transition matrix from the
elementary symmetric function $S_{-m}^*.1$ to the generalized symmetric functions $q_{\rho}$. Using the definition of $Q(z)$, it is easy to see that
\begin{equation}
S^*(tz).1=Q(z)S^*(z).1
\end{equation}
Then we have the generalized Newton's formula:
\begin{equation}
S^*_{-m}.1=\frac1{t^m-1}(q_1S^*_{-m+1}.1+q_2S^*_{-m+2}.1+\cdots q_{m-1}S^*_{-1}.1+q_m),
\end{equation}
which leads to the decomposition: $S_{-m}^*.1=\sum_{\rho\vdash m}C_{m,\rho}(t)q_{\rho}$. 
The first few terms are given by
\begin{align*}
S^*_{-1}.1&=\frac{q_1}{t-1},\\
S^*_{-2}.1&=\frac{q_1^2}{(t^2-1)(t-1)}+\frac{q_2}{t^2-1},\\
S^*_{-3}.1&=\frac{q_1^3}{(t^3-1)(t^2-1)(t-1)}+\frac{(t+2)q_2q_1}{(t^3-1)(t^2-1)}+\frac{q_3}{t^3-1},\\
S^*_{-4}.1&=\frac{q_1^4}{(t^4-1)(t^3-1)(t^2-1)(t-1)}+\frac{(t^2+t+2)q_3q_1}{(t^4-1)(t^3-1)}\\
&+\frac{(t^2+2t+3)q_2q_1^2}{(t^4-1)(t^3-1)(t^2-1)}+\frac{q_2^2}{(t^4-1)(t^2-1)}+\frac{q_4}{t^4-1}.
\end{align*}

The following result can be shown similarly as Theorem \ref{t:gen}.
\begin{thm} For $\la, \mu\vdash n$. Then the irreducible character $\chi^{\la}_{\mu}(q)$ is given by
\begin{align*}
\chi^{\la}_{\mu}(q)=\frac{q^{\la_1}}{(q-1)^{l(\mu)}}\sum_{i=\la_1}^n\sum_{\mbox{\tiny$\begin{matrix}|\tau|=i\\
\tau\subset\mu\end{matrix}$}}(1-q^{-1})^{l(\tau)}
\sum_{\rho\vdash i-\la_1}(q-1)^{l(\mu-\tau)+l(\rho)}C_{i-\la_1, \rho}(q^{-1})\chi^{\la^{[1]}}_{(\mu-\tau)^*\cup \rho}(q),
\end{align*}
where $C_{m,\rho}(t)$ is the transition coefficient from $S^*_{-m}.1$ to $q_{\rho}$ and $(\mu-\tau)^*$ is the rearranged partition obtained from $\mu-\tau$.
\end{thm}

\section{Applications of the hook and two-row formulas}
Let $V=V_{\bar{0}}\bigoplus V_{\bar{1}}$ be a $\mathbb{Z}_2$-graded vector space over $\mathbb{C}(q)$ with basis $\{v_1,v_2,\cdots,v_a\}$ for the even subspace $V_{\bar 0}$ and basis $\{v_{a+1},v_{a+2},\cdots,v_{a+b}\}$ for the odd subspace $V_{\bar 1}$. Let $\pi$ be the endomorphism of $V\otimes V$ by 
\begin{align}
(v_k\otimes v_l)\pi=
\begin{cases}
(-1)^{|v_k||v_l|}v_l\otimes v_k+(q-1)v_k\otimes v_l, k<l\\
\frac{(-1)^{|v_k|}(q+1)+q-1}{2}v_k\otimes v_l, k=l\\
(-1)^{|v_k||v_l|}tv_l\otimes v_k, k>l,
\end{cases}
\end{align}
and let $\pi_i$ be the endomorphism of $V^{\otimes n}$ by letting $\pi$ acting on the $(i, i+1)$ factor and identity elsewhere, $i=1, \cdots, n-1$.
It is known \cite{Mit} that the map $\Phi^{q,n}_{a,b}: H_n(q)\longrightarrow \mathrm{End}_{\mathbb{C}(t)}(V^{\otimes n})$ defined by
\begin{align*}
T_i\longmapsto\pi_i
\end{align*}
gives rise to a representation of $H_n(q)$ called the sign $q$-permutation representation.

Let $\varphi^{n}_{a,b}$ be the character of $\Phi^{q,n}_{a,b}$ and let $s_{a,b}(\lambda)$ be the number of $(a,b)$-semi standard tableaux of shape $\lambda$ \cite{Sag}. Then $\varphi^{n}_{a,b}$ decomposes as follows \cite{Zhao}:
\begin{align*}
\varphi^{n}_{a,b}=\sum\limits_{\lambda\in H(a,b;n)}s_{a,b}(\lambda)\chi^{\lambda}
\end{align*}
where $H(a,b;n)=\{\lambda\vdash n\mid \lambda_{a+1}\leq b\}$ and $\chi^{\la}$ is the irreducible character of $H_n(q)$ labeled by $\la\vdash n$.

We consider two special cases $a=b=1$ and $a=2, b=0$. Then
\begin{align}\label{e:1,1}
\varphi^{n}_{1,1}&=\sum\limits_{k=0}^{n-1}\chi^{(n-k,1^k)}\\\label{e:2,0}
\varphi^{n}_{2,0}&=\sum\limits_{k=0}^{[n/2]}(n-2k+1)\chi^{(n-k,k)}
\end{align}

\begin{thm} For $\mu\vdash n$ with $l(\mu)=l$, we have that
\begin{align}\label{e:hooksum}
\varphi^n_{1,1}(T_{\gamma_{\mu}})&=\sum\limits_{i=0}^{n-1}\chi^{(n-i,1^i)}_{\mu}(q)
=(-1)^{n-l}2^{l-1}\prod\limits_{i=1}^{l}[\mu_i]_{-q},
\\\label{e:twosum}
\varphi^n_{2,0}(T_{\gamma_{\mu}})&=\sum\limits_{i=0}^{[n/2]}(n-2i+1)\chi^{(n-i,i)}_{\mu}(q)
=q^{n-2l(\mu)}\prod\limits_{i=1}^{l}(1+q+\mu_i(q-1)).
\end{align}
\end{thm}
\begin{proof} It follows from the hook formula \eqref{e:hook} that
\begin{align*}
\sum\limits_{i=0}^{n-1}\chi^{(n-i,1^i)}_{\mu}(q)
=&\sum\limits_{i=0}^{n-1}(-1)^{l+i}\sum\limits_{j=n-i}^{n}a_{j}(\mu;q)q^{j}\\
=&(-1)^{n+l-1}\sum\limits_{j=1}^{[n/2]}a_{2j-1}(\mu;q)q^{2j-1}\\
=&(-1)^{n-l}\sum\limits_{j=0}^{n}\frac{(-1)^j-1}{2}a_{j}(\mu;q)q^j\\
=&(-1)^{n-l}2^{l-1}\prod\limits_{i=1}^{l}[\mu_i]_{-q},
\end{align*}
where we have used Lemma \ref{l:ab} and the identity $\sum_{i=0}^na_i(\mu; q)q^i=0$.

Similarly, using the two-row formula \eqref{e:two} and \eqref{e:bv} we have that
\begin{align*}
&\sum\limits_{i=0}^{[n/2]}(n-2i+1)\chi^{n-i,i}_{\mu}(q)\\
=&\sum\limits_{i=0}^{[n/2]}(n-2i+1)q^{n-l}(b_{n-i}(\mu;q)-b_{n-i+1}(\mu;q))\\
=&q^{n-l}\sum_{i=0}^nb_i(\mu; q)\\
=&q^{n-2l(\mu)}\prod\limits_{i=1}^{l}(1+q+\mu_i(q-1)).
\end{align*}
\end{proof}

We remake that Zhao \cite{Zhao} obtained a negated version of
\eqref{e:hooksum} as a $q$-analog Berele-Regev formula \cite{BR} by the Murnaghan-Nakayama rule.

\section{Determinant type Murnaghan-Nakayama rule}
The Murnaghan-Nakayama rule is an iterative formula to compute the Hecke algebra characters. The combinatorial rule
was proved by Ram \cite{Ram} using the Frobenius formula and the Hall-Littlewood symmetric functions. In this section, we formulate a determinant type Murnaghan-Nakayama rule and use it to compute the characters and also give a new proof of the Murnaghan-Nakayama rule.

For two partitions $\mu\subset\lambda,$ the set-theoretic difference $\theta=\lambda-\mu$ is called a {\it skew diagram} denoted by $\lambda/\mu$. A subset $\xi$ of $\theta$ is connected if any two squares in $\theta$ are connected by a path in $\xi$. The connected components, themselves skew diagrams, are by definition the maximal connected subsets of $\theta$.
A skew diagram $\lambda/\mu$ is a vertical (resp. horizontal) strip if each row (resp. column) contains at most one box. A skew diagram $\theta$ is a {\it border strip} if it contains no $2\times 2$ blocks of squares and is connected (see \cite{Mac}).
If the border strip $\theta=\la-\mu$ has $m$ connected components $(\xi_1,\xi_2,\ldots,\xi_m)$, we define the weight of $\theta$ by
$$wt(\theta;t)= (t-1)^{m-1}\prod\limits_{i=1}^{m}(-1)^{r(\xi_{i})-1}t^{c(\xi_i)-1}.$$
where the $r(\xi_i)$ (resp. $c(\xi_i)$) means the number of rows (resp. columns) in the border strip $\xi_i$. Usually we refer $\theta$ as a broken border strip to emphasize the non-connectedness of $\theta$.

Let $\lambda\vdash n$. A Young tableau of shape $\lambda$ is an assignment of numbers $1,2,\ldots,n$ into the dots of the Young diagram.
A tableau is {\it standard} if its rows and columns are increasing sequences. Let $f^{\lambda}$ be the number of standard tableau of shape $\lambda$ and let $\lambda^{-}$ be a partition obtained by removing an inner corner of $\lambda,$ then \cite{Sag}
\begin{align}\label{e:branch}
f^{\lambda}=\sum\limits_{\lambda^{-}}f^{\lambda^{-}}.
\end{align}

Recall from \eqref{e:qS}, for a fixed partition $\lambda$ with $l=l(\lambda)$ we have that
\begin{align}\label{e:qS2}
q^*_{k}S_{\lambda}.1=\sum_{\mbox{\tiny$\begin{array}{c}
\tau_1,\ldots, \tau_l\geq0\\
|\tau|=k\end{array}$}}(1-t)^{l-\sum\delta_{\tau_i,0}}S_{\lambda-\tau}.1.
\end{align}

For any composition $\tau\vDash k$, $S_{\lambda-\tau}.1=sg(\sigma)S_{\mu}.1$ if
$\la-\tau+\delta$ is strict and $\mu=\sigma(\lambda-\tau+\delta)-\delta$ for some permutation $\sigma\in \mathfrak S_l$, otherwise $S_{\lambda-\tau}.1=0$.
Note that $\mu$ must be a partition $\subset\la$ with weight $|\la|-k$ if $S_{\lambda-\tau}.1\neq 0$. 
In fact, note that $\mu_i=\lambda_{\sigma(i)}-\tau_{\sigma(i)}+i-\sigma(i)$ for some permutation $\sigma$ that sends $\lambda-\tau+\delta$ into partition $\mu$.
If $j=\sigma(i)\geq i$, then $\lambda_i-\mu_i=\lambda_i-\lambda_j+\tau_j-i+j\geq 0$. If $j=\sigma(i)<i$, we claim that $\la_i-\mu_i\geq 0$ otherwise $\lambda_i-\mu_i=\lambda_i-\lambda_j+\tau_j-i+j<0$, then
$$(\lambda_j-\tau_j, \lambda_{j+1}-\tau_{j+1}, \ldots, \lambda_i-\tau_i)+(i-j, \ldots, 1, 0)=(\lambda_j-\tau_j+i-j, \ldots, \lambda_i-\tau_i).$$
Observe that $(\lambda_i-\tau_i)-(\lambda_j-\tau_j+i-j)<0$
so $\sigma$ cannot send $i$ to $j$, which is a contradiction!

Conversely, for any $\mu\vdash |\la|-k$ and any $\sigma\in \mathfrak S_l$, the $l$-tuple
\begin{equation*}
(\la_1-\mu_{\sigma(1)}-1+\sigma(1), \la_2-\mu_{\sigma(2)}-2+\sigma(2), \cdots, \la_l-\mu_{\sigma(l)}-l+\sigma(l))
\end{equation*}
is a composition of $k$ provided that $\la_i-\mu_{\sigma(i)}+i-\sigma(i)\geq 0$. Then
$sg(\sigma)S_{\mu}.1=S_{\la-\tau}.1$ for $\tau_i=\la_i-\mu_{\sigma(i)}-i+\sigma(i)$. Therefore \eqref{e:qS2} can be rewritten as:
\begin{align*}
q^*_{k}S_{\lambda}.1&=\sum\limits_{\mu}\sum_{\mbox{\tiny$\begin{array}{c}
\sigma\in  S_l\\
\lambda_i-i-\mu_{\sigma(i)}+\sigma(i)\geq 0\end{array}$}}sgn(\sigma)(1-t)^{l-\sum\delta_{\lambda_i-i-\mu_{\sigma(i)}+\sigma(i),0}}S_{\mu}.1\\
&=(1-t)^{l}\sum\limits_{\mu}\det(\delta_{\lambda_i-i\geq \mu_j-j}(1-t)^{-\delta_{\lambda_i-i, \mu_j-j}})S_{\mu}.1
\end{align*}
where the sum is over all partitions $\mu\subset \lambda$ such that $\mid\lambda/\mu\mid=k$.

For partitions $\mu\subset\lambda$ and $|\la/\mu|=k$, consider the $l(\la)\times l(\la)$-matrix
$$M(\lambda/\mu; t)=(\delta_{\lambda_i-i\geq\mu_j-j}(1-t)^{-\delta_{\lambda_i-i, \mu_j-j}})$$
where $\delta_{a\geq b}=1$ if $a\geq b$ and $\delta_{a\geq b}=0$ otherwise. Then we have proven the following theorem.
\begin{thm}
Let $\lambda=(\lambda_{1},\ldots, \lambda_{l})$ be a partition and $k$ be a positive integer. Then one has
\begin{align}\label{e:de}
q^*_{k}S_{\lambda}.1=\sum\limits_{\mu}(1-t)^{l}\det M(\lambda/\mu; t)S_{\mu}.1
\end{align}
where the sum is over all partitions $\mu\subset \lambda$ such that $\mid\lambda/\mu\mid=k$.
\end{thm}

The matrix $M(\la/\mu; t)$ has the following properties: (i) if $M(\la/\mu; t)_{ij}=\frac1{1-t}$, then all entries in the southwest region to the $(i, j)$-entry are zero; (ii) if $M(\la/\mu; t)_{ij}=0$, then all entries down the $j$th column are also zero;
and (iii) if $M(\la/\mu; t)_{ij}=1$, then all entries in the northeast region to the $(i, j)$-entry are also $1$.
et's divide $\lambda/\mu$ into several cases:

{\bf Case $1:$} If $\lambda/\mu$ is a $k$-border strip with the initial box at the
$r$th row, 
then $M(\la/\mu; t)$ has the following form: 
\begin{gather*}
\begin{pmatrix}
\frac1{1-t}  & \cdots & 1 & 1& 1 & \cdots & 1 & 1 \\
\vdots           & \ddots & \vdots  & \vdots & \vdots & \cdots & \vdots & \vdots \\
0 & \cdots & \frac1{1-t}  & 1 & 1 & \cdots & 1 & 1\\
0 & \cdots & 0 &  1  & 1 & \cdots & 1 & 1 \\
0 &\cdots & 0 & \frac{1}{1-t}  & 1 &\cdots & 1 &1\\
0 &\cdots & 0 &0 & \frac{1}{1-t} & \cdots & 1 &1\\
\vdots &\vdots & \vdots &\vdots &\vdots &\ddots &\vdots &\vdots\\
0 &\cdots & 0 & 0 &0 &\cdots &\frac{1}{1-t} &1
\end{pmatrix}_{l\times l}
=\begin{pmatrix}M_0 & * \\
 0 & M_1
\end{pmatrix}, \\
\det(\lambda/\mu;k)=(1-t)^{-r+1}(\frac{-t}{1-t})^{l-r}
\end{gather*}
where $M_0$ is an $(r-1)\times (r-1)$ upper-triangular matrix with $\frac1{1-t}$ on the diagonal and $M_1$ is a $(l-r+1)\times (l-r+1)$ quasi upper-triangular matrix
with $1$ on the diagonal and $\frac1{1-t}$ on the (lower) secondary diagonal. For simplicity we denote $M_1$ by $M(\xi; t)$, where
$\xi$ is the border strip obtained from $\la/\mu$ by removing the initial empty boxes.

{\bf Case $2:$} If $\lambda/\mu$ is connected and contains a $2 \times 2$ block of boxes, then
$\det(\la/\mu; t)=0$. This can be seen as follows. Without loss of generality, suppose $ht(\lambda/\mu)=l(\lambda)-1$ and the $2 \times 2$ block appear at the $r$-th and $(r+1)$-th rows $(1\leq r\leq l-1)$, then $M(\la/\mu; t)$ has the following form:
\begin{gather*}
M(\lambda/\mu;t)=
\bordermatrix{%
& & & &r-1 &r & & & \cr
&1  &1 &\cdots &1 &1 &\cdots &1 &1\cr
&\frac{1}{1-t}  &1 &\cdots &1 &1 &\cdots &1 &1\cr
&0 &\frac{1}{1-t} &\cdots &1 &1 &\cdots &1 &1\cr
&\vdots &\vdots &\ddots &\vdots &\vdots &\ddots &\vdots &\vdots\cr
r &0 &0 &\cdots &\frac{1}{1-t} &1 &\cdots &1 &1\cr
r+1 &0 &0 &\cdots &0 &1 &\cdots &1 &1\cr
&\vdots &\vdots &\ddots &\vdots &\vdots &\ddots &\vdots &\vdots\cr
&0 &0 &0 &0 &0 &\cdots &\frac{1}{1-t} &1
}, \quad \det(\lambda/\mu;t)=0.
\end{gather*}

{\bf Case $3:$} If $\lambda/\mu$ has two connected components $\xi_1, \xi_2$, then $M(\lambda/\mu; t)$
is a $2\times 2$ upper-triangular block matrix with the diagonal blocks corresponding to $\xi_1, \xi_2$. If
one of the components contains a $2\times 2$-block of boxes, then $\det(\la/\mu; t)=0$. Suppose each
$\xi_i$ is a border strip and there are $r_i$ unoccupied rows immediately above $\xi_i$ on the border of $\lambda$, then $M(\la/\mu; t)$ has the following form:
\begin{gather*}
M(\la/\mu; t)=\begin{pmatrix}M_1 & * \\
 0 & M_2
\end{pmatrix}, \qquad \det M_i=(1-t)^{-r_i+1}(\frac{-t}{1-t})^{ht(\xi_i)-1},
\end{gather*}
where each $M_i$ is a $2\times 2$ block diagonal matrices specified in case one.

In general, if $\lambda/\mu$ has $m$ connected components $\xi_1, \xi_2,\cdots,\xi_m$, then $\det M(\la/\mu; t)=0$
unless each $\xi_i$ is a border strip. Then
\begin{gather*}
M(\lambda/\mu;t)=
\begin{pmatrix}
M_1  &* &\cdots &*\\
0  &M_2 &\cdots &*\\
0 &\vdots &\cdots &\vdots\\
0 &0 &\vdots &M_s
\end{pmatrix}
\end{gather*}
where the diagonal blocks are either $\frac{1}{1-t}$ or the quasi upper-triangular matrix $M(\xi_j;t)$, $i=1,2,\ldots,s; j=1,2\ldots,m$.
It is clear that the number of $\frac1{1-t}$ on the diagonal is equal to $Card\{\lambda_i=\mu_i\mid i=1,\ldots,l\}=l-\sum\limits_{j} r(\xi_j)$ and the order of $M(\xi_j;k)$ is equal to $r(\xi_j), j=1,2, \ldots, m$.
Therefore,
\begin{align*}
\det(\lambda/\mu;t)&=(\frac{1}{1-t})^{l-\sum r(\xi_i)}\prod\limits_{i=1}^{m}\det(\xi_{i};t)\\
&=(\frac{1}{1-t})^{l-\sum r(\xi_i)}\prod\limits_{i=1}^{m}(\frac{-t}{1-t})^{r(\xi_i)-1}\\
&=(1-t)^{m-l}\prod\limits_{i=1}^{m}(-t)^{r(\xi_i)-1}.
\end{align*}
Thus we have proved the following result.
\begin{cor} \label{t:det}
Let $\lambda$ be a partition and $k$ a positive integer. Then one has that
\begin{align}
q^*_{k}S_{\lambda}.1=\sum\limits_{\mu}(1-t)^{C(\la/\mu)}\prod\limits_{i=1}^{C(\la/\mu)}(-t)^{r(\xi_i)-1}S_{\mu}.1
\end{align}
where the sum runs through all partitions $\mu\subset \lambda$ such that $\lambda/\mu$ is a $k$-broken border strip.
Here $\xi_1,\ldots,\xi_{C(\la/\mu)}$ are the connected components of $\la/\mu$.
\end{cor}

\begin{rem} If $k=|\lambda|=n$ in Corollary \ref{t:det}, then $q^*_{k}S_{\lambda}.1=0$ unless $\lambda$ is a hook. In this case,  $q^*_{n}S_{(m,1^{n-m})}.1=(1-t)(-t)^{n-m}$. So for all $\lambda\vdash n\geq1$,
\begin{align*}
g^{\lambda}_{(n)}(t)=
\begin{cases}
0,&\text{if $\lambda$ is not a hook}\\
(1-t)(-t)^{n-m},&\text{if $\lambda=(m,1^{n-m})$}.
\end{cases}
\end{align*}
Similarly, when $k=1$, then $q^*_1S_{\lambda}.1=(1-t)\sum\limits_{\lambda^{-}}S_{\lambda^{-}}.1$, where
$\lambda^{-}$ runs through all partitions obtained from $\la$ by removing an inner corner. Therefore,
\begin{align*}
g^{\lambda}_{(1^n)}(t)=(1-t)\sum\limits_{\lambda^{-}}g^{\lambda^{-}}_{(1^{n-1})}(t).
\end{align*}
By \eqref{e:branch} it follows that,
\begin{align*}
g^{\lambda}_{(1^n)}(t)=(1-t)^nf^{\lambda}
\end{align*}
where $f^{\lambda}$ is the number of standard tableau of shape $\lambda$.
\end{rem}

We now present a new proof of the Murnaghan-Nakayama rule for $H_n(q)$ \cite{Ram} (cf. \cite{HR, Ha}).
\begin{thm}{\bf (The Murnaghan-Nakayama rule for $H_n(q)$)}\label{t:M-N rule}
Let $\mu$ be a partition and $k$ a positive integer. Then one has that
\begin{align}
\tilde{q}_{k}S_{\mu}.1=\sum\limits_{\lambda}wt(\lambda/\mu;q)S_{\lambda}.1,
\end{align}
where the sum is over all partitions $\la\supset\mu$   
such that $\lambda/\mu$ is a $k$-broken border strip.
\end{thm}
\begin{proof} By Corollary \ref{t:det} and duality, it is enough to check that
\begin{align*}
(1-t)^{m}\prod\limits_{i=1}^{m}(-t)^{r(\xi_i)-1}=t^{k-1}(1-t)wt(\lambda/\mu;t)
\end{align*}
where $\lambda/\mu$ is a $k$-broken border strip with $m$ connected components $\xi_1,\ldots,\xi_m$.
Suppose $\mid\xi_i\mid=k_i,$ then $\sum_i k_i=k.$ Since $\xi_i$ is a border strip, $r(\xi_i)+c(\xi_i)=k_i+1$, $i=1,2,\ldots, m$. Therefore,
\begin{align*}
t^{k-1}(1-t)wt(\lambda/\mu;t)&=(1-t)^mt^{k-m}\prod\limits_{i=1}^{m}(-1)^{r(\xi_i)-1}t^{1-c(\xi_i)}\\
&=(1-t)^m\prod\limits_{i=1}^{m}(-1)^{r(\xi_i)-1}t^{k_i-c(\xi_i)}\\
&=(1-t)^{m}\prod\limits_{i=1}^{m}(-t)^{r(\xi_i)-1}.
\end{align*}
\end{proof}

\section{The bitrace of the regular representation}
The Hecke algebra $H_n(q)$ is a $(H_n(q), H_n(q))$-bimodule under the left and right regular actions, and the two actions mutually commute with each other. As  a result
\begin{equation}
H_n(q)=\bigoplus_{\lambda\vdash n} V_{\lambda}\otimes V^{\lambda}
\end{equation}
where $V_{\lambda}$ (resp. $V^{\lambda}$) is the irreducible left (resp. right) $H_n(q)$-module labeled by $\lambda$.

Following \cite{HLR}, we introduce the bitrace of the regular representation of $H_n(q)$ by defining for any compositions
$\lambda, \mu\models n$
\begin{align}
btr(\lambda, \mu)=\sum\limits_{\rho\vdash n}\chi^{\rho}_{\lambda}(q)\chi^{\rho}_{\mu}(q).
\end{align}
This is a $q$-deformation of the second orthogonality relation between irreducible characters of the symmetric group. It is known that $btr(\lambda, \mu)|_{q=1}=\delta_{\lambda, \mu}z_{\lambda}$.

Halverson, Luduc and Ram \cite{HLR} proved the following result by using Roichman's formula \cite{Ro} for the Hecke algebra.
We will give an elementary proof using the technique developed in this paper.

\begin{thm}\label{t:brt}
Let $\lambda=(\lambda_1,\lambda_2,\ldots,\lambda_r)$ and $\mu=(\mu_1,\mu_2,\ldots,\mu_s)$ be two compositions of $n$. Then
\begin{align}\label{e:brt}
btr(\lambda, \mu)=(q-1)^{-r-s}\sum\limits_M wt(M),
\end{align}
where $M=(m_{ij})$ run through all $r\times s$ nonnegative integral matrices
such that $\sum_{a}m_{ia}=\la_i$ $(1\leq i\leq r)$ and $\sum_bm_{bj}=\mu_j$ $(1\leq j\leq s)$ and
\begin{align*}
wt(M)=\prod\limits_{m_{ij}\neq 0}(q-1)^2[m_{ij}]_{q^2}.
\end{align*}
\end{thm}

By the Frobenius formula of the Hecke algebra $H_n(q)$ it is readily seen that
\begin{align*}
btr(\lambda, \mu)=\langle \tilde{q}_{\lambda}(q), \tilde{q}_{\mu}(q) \rangle.
\end{align*}

 Let $H^{\lambda}_{\mu}(t)=\langle q_{\lambda}(t), q_{\mu}(t) \rangle$, then $brt(\lambda, \mu)(q)=\frac{q^{2n}}{(q-1)^{l(\lambda)+l(\mu)}}H^{\lambda}_{\mu}(q^{-1})$.

Using vertex operators and \eqref{e:qop}-\eqref{e:qop*}, we have the following relation:
\begin{align}
Q^*(z)Q(w)\frac{z-t^2w}{z-tw}=Q(w)Q^*(z)\frac{z-tw}{z-w},
\end{align}
where the rational functions are power series in negative powers of $z$.
Comparing coefficients of $z^{-n}w^{m}$ we immediately obtain the following commutation relation.
\begin{prop} For any $m,n\in\mathbb{Z}_+$, as operators on $\Lambda_{\mathbb C[t, t^{-1}]}$
\begin{align}\label{e:com3}
q^{*}_{n}q_{m}=q_{m}q^{*}_{n}+(1-t)\sum\limits_{k=1}^{min(m, n)}q_{m-k}q^*_{n-k}+(t-1)\sum\limits_{k=1}^{min(m, n)}q^*_{n-k}q_{m-k}t^k.
\end{align}
\end{prop}

For simplicity we denote for integer $k>0$
$$(k)_{t}=(t-1)\frac{t^{2k}-1}{t+1}=(t-1)^2[k]_{t^2}$$
and make the convention that $(k)_t=0$ for $k<0$ and $(0)_t=1$.
The following result is easily shown by induction, and also can be used as an inductive definition of $(k)_t$.
\begin{lem}\label{t:idetity}
Let $k$ be a positive integer, then we have
\begin{align}\label{e:identity}
1-t+(t-1)\sum\limits_{i=1}^{k}(k-i)_tt^i=(k)_t.
\end{align}
\end{lem}

\begin{thm}\label{t:q*q}
Let $k\in \mathbb Z_+$ and $\mu\vDash n$ a composition, then
\begin{align}
q^*_{k}q_{\mu}=\sum\limits_{\tau\models k}\prod\limits_{\tau_a>0}(\tau_a)_t q_{\mu-\tau}.
\end{align}
\end{thm}
\begin{proof} Use induction on $k+n$. The initial step is clear. Assume the identity holds for $q^*_{k'}q_{\mu}$ such that $k'+|\mu|<k+n$, then
\begin{align*}
&q^*_{k}q_{\mu_1}q_{\mu_2}\cdots q_{\mu_l}\\
=&q_{\mu_1}q^*_{k}q_{\mu_2}\cdots q_{\mu_l}+(1-t)\sum\limits_{i\geq1}q_{\mu_{1}-i}q^*_{k-i}q_{\mu_2}\cdots q_{\mu_l}+(t-1)\sum\limits_{i\geq1}q^*_{k-i}q_{\mu_{1}-i}q_{\mu_2}\cdots q_{\mu_l}t^i\\
=&q_{\mu_1}\sum\limits_{\tau\models k}\prod\limits_{\tau_a\geq1}(\tau_a)_t q_{\mu^{[1]}-\tau}+(1-t)\sum\limits_{i\geq1}\sum\limits_{
\tau\models k-i}\prod\limits_{\tau_a\geq 1}(\tau_a)_t q_{\mu_{1}-i}q_{\mu^{[1]}-\tau}\\
&\hskip 3cm +(t-1)\sum\limits_{i\geq1}\sum\limits_{\tau\models k-i}\prod\limits_{\tau_a\geq1}(\tau_i)_tq_{\mu_{1}-i-\tau_1}q_{\mu^{[1]}-\tau^{[1]}}t^i.
\end{align*}
The first summand is $\sum\limits_{\tau\models k}\prod\limits_{\tau_a\geq1}(\tau_a)_t q_{\mu-\tau}$ with
$\tau_1=0$. Replacing $i+\tau_1$ by $j$ in the third summand and pulling $(\tau_1)_t=(j-i)_t$ forward, we simplify the second and
the third summands as follows:
\begin{align*}
&(1-t)\sum\limits_{i\geq1}\sum\limits_{
\tau\models k-i}\prod\limits_{\tau_a\geq1}(\tau_a)_t q_{\mu_{1}-i}q_{\mu^{[1]}-\tau}+(t-1)\sum\limits_{i\geq1}\sum\limits_{j\geq i}(j-i)_tt^i\sum\limits_{\tau\models k-j}\prod\limits_{\tau_a\geq1}(\tau_a)_tq_{\mu_{1}-j}q_{\mu^{[1]}-\tau}\\
=&(1-t)\sum\limits_{i\geq1}\sum\limits_{
\tau\models k-i}\prod\limits_{\tau_a\geq1}(\tau_a)_t q_{\mu_{1}-i}q_{\mu^{[1]}-\tau}+\sum\limits_{j\geq1}(t-1)\sum\limits_{i=1}^{j}(j-i)_tt^i\sum\limits_{\tau\models k-j}\prod\limits_{\tau_a\geq1}(\tau_a)_tq_{\mu_{1}-j}q_{\mu^{[1]}-\tau}\\
=&\sum\limits_{j\geq1}(j)_t\sum\limits_{\tau\models k-j}\prod\limits_{\tau_a\geq1}(\tau_a)_tq_{\mu_{1}-j}q_{\mu^{[1]}-\tau}  ~~~\text{(by Lemma \ref{t:idetity})}.
\end{align*}
Combining this with the first summand, we see the total sum is $\sum\limits_{\tau\models k}\prod\limits_{\tau_a\geq1}(\tau_a)_t q_{\mu-\tau}$.
\end{proof}

The following is an immediate consequence of Theorem \ref{t:q*q}. 
\begin{cor}Let $\lambda, \mu\models n$ with $l(\lambda)=r$ and $l(\mu)=s$. Then
\begin{align}\label{e:H}
H_{\mu}^{\lambda}(t)=\sum\limits_{M}\prod\limits_{i=1}^{r}\prod\limits_{a\geq1}(\tau^{(i)}_a)_t.
\end{align}
where the sum is over all $r\times s$ nonnegative integer matrices $M=(\tau^{(i)}_a)_{1\leq i\leq r, 1\leq a \leq s}$ such that
the row sums are $\lambda_1,\ldots,\lambda_r$ and the column sums are $\mu_1,\ldots,\mu_s$ 
\end{cor}
This provides a simple proof of Theorem \ref{t:brt} due the fact that
$t^{2k}(k)_{t^{-1}}=(k)_t,$ and $q_i=0$ for $i<0$.


\vskip30pt \centerline{\bf Acknowledgments}
The work is partially supported by
Simons Foundation grant No. 523868 and NSFC grant No. 11531004.
\bigskip

\bibliographystyle{plain}

\end{document}